\documentclass[12pt]{amsart}
\usepackage{amsmath, amsthm, amscd, amsfonts, amssymb, graphicx, color}
\RequirePackage[backref,colorlinks=true]{hyperref}
\newcommand{\R}{{\mathbb{R}}}
 
 \newcommand{\inter}{\mathopen [0,1\mathclose]}
\newcommand{\interu}{\mathopen [u\sp{-}(\lambda),
u\sp{+}(\lambda)\mathclose]}
\newcommand{\interv}{\mathopen [v\sp{-}(\lambda),
v\sp{+}(\lambda)\mathclose]}

 \newcommand{\cj}[2]{\left \{\, #1\, : \, #2\, \right\}}
 
 \newcommand{\n}{\noindent}

\newcommand{\fn}{\mathbb{E}\sp{1}}

\begin{document}

\title[Best approximation]{Best approximation results for fuzzy-number-valued continuous functions}


\author{Juan J. Font}
\address{Institut Universitari
de Matem\`{a}tiques i Aplicacions de Castell\'{o} (IMAC),
Universitat Jaume I, Avda. Sos Baynat s/n, 12071, Castell\'{o}
(Spain).}
\thanks{This research received no funding}
\email{font@uji.es}

\author{Sergio Macario}

\address{Institut Universitari
de Matem\`{a}tiques i Aplicacions de Castell\'{o} (IMAC),
Universitat Jaume I, Avda. Sos Baynat s/n, 12071, Castell\'{o}
(Spain).}
\email{macario@uji.es}



\date{\today}

%


%



\begin{abstract}
In this paper we study the best approximation of a fixed fuzzy-number-valued continuous function to a subset of fuzzy-number-valued continuous functions. We also introduce a method to measure the distance between a fuzzy-number-valued continuous function and a real-valued one. Then we prove the existence of the best approximation of a fuzzy-number-valued continuous function to the space of real-valued continuous functions by using the well-known Michael Selection Theorem.
\end{abstract}

\maketitle



\newtheorem{theorem}{Theorem}[section]
\newtheorem{lemma}[theorem]{Lemma}
\newtheorem{proposition}[theorem]{Proposition}
\newtheorem{corollary}[theorem]{Corollary}
\newtheorem{question}[theorem]{Question}

\theoremstyle{definition}
\newtheorem{definition}[theorem]{Definition}
\newtheorem{algorithm}[theorem]{Algorithm}
\newtheorem{conclusion}[theorem]{Conclusion}
\newtheorem{problem}[theorem]{Problem}

\theoremstyle{remark}
\newtheorem{remark}[theorem]{Remark}
\numberwithin{equation}{section}

\section{Introduction}

Approximation Theory originated from the necessity of approximating real-valued continuous
functions by simpler class of functions, such as trigonometric or algebraic polynomials and has attracted
the interest of many mathematicians for over a century. Among the most recognized results in this branch of Functional Analysis, we can mention the Stone-Weierstrass Theorem, Korovkin type results and the approximation of functions using neural networks.

More recently, all the above results have been also addressed in the context of fuzzy functions (see, e.g., \cite{anas}, \cite{huangwu}, \cite{FSS17}).

Another fundamental problem in Approximation Theory
is the study of best approximation in spaces of continuous functions, which has a long story with famous results by Chebyshev, Haar, Young, Remez, de la Vallée-Poussin who established the existence of best approximations, as well as characterized and estimated them. Here we could mention the problem of the uniform approximation of a scalar-valued
function continuous on a compact set by a set of other functions continuous on this compact set (see, e.g., \cite{prolacamb}, \cite{prola} or \cite{chen}).

The existence of the best approximation of a set-valued continuous function by vector-valued ones is another important topic in Approximation Theory and has been studied by several authors (see, e.g., \cite{Olech}, \cite{prolamachado}, \cite{lau}, \cite{cuenya} or \cite{kashi}).

In this paper we address these two problems of best approximation type in the context of fuzzy-number-valued continuous functions.

First, we study the best approximation of a fixed fuzzy-number-valued continuous function to a subset of fuzzy-number-valued continuous functions.

Secondly, we introduce a novel method to measure the distance between a fuzzy-number-valued function and a real-valued one. Then we prove the existence of the best approximation of a fuzzy-number-valued continuous function to the space of real-valued continuous functions by using the well-known Michael Selection Theorem.

\section{Preliminaries}


Let $F(\R)$ denote the family of all fuzzy subsets on the real
numbers $\R$ (see \cite{DK2}). For $\lambda\in\inter$ and a fuzzy set $u$, its $\lambda$-level set
is defined as

$$
[u]\sp{\lambda}:=\cj{x\in\R}{u(x)\geq \lambda}, \quad
\lambda\in ]0,1],
$$
and $[u]\sp{0}$ stands for the closure of $\cj{x\in\R}{u(x)>0}$.

The set of elements $u$ of
$F(\R)$ satisfying the following properties:

\begin{enumerate}
 \item $u$ is normal, i.e., there exists an $x\sb{0}\in\R$ with $u(x\sb{0})=1$;

\item $u$ is convex, i.e., $u(\lambda x + (1-\lambda)y)\geq \min
\left\{u(x),u(y)\right\}$  for all $x,y\in \R, \lambda\in\inter$;

\item $u$ is upper-semicontinuous;

\item $[u]\sp{0}$ is a compact set in $\R$,
\end{enumerate}
is called the fuzzy number space $\fn$ and contains the reals. 
 If $u\in\fn$, then it is known that $\lambda$-level set
$[u]\sp{\lambda}$ of $u$ is a compact interval for each
$\lambda\in\inter$. We will write $[u]\sp{\lambda}=\interu$.  

The following characterization of fuzzy numbers, which was proved by Goetschel and Voxman (\cite{GV}), is essential in the sequel:

\begin{theorem}\label{GW1}
Let $u\in\fn$ and $[u]\sp{\lambda}=\interu$, $\lambda\in\inter$.
Then the pair of functions $u\sp{-}(\lambda)$ and
$u\sp{+}(\lambda)$ has the following properties:

\begin{enumerate}
\item  $u\sp{+}(\lambda)$ is a bounded left continuous nonincreasing function
on $\mathopen (0,1\mathclose]$;

\item $u\sp{-}(\lambda)$ is a bounded left continuous nondecreasing function
on $\mathopen (0,1\mathclose]$;

\item $u\sp{-}(1)\leq u\sp{+}(1)$;

\item $u\sp{-}(\lambda)$ and $u\sp{+}(\lambda)$ are right continuous at
$\lambda=0$.
\end{enumerate}

\n Conversely, if two functions $\gamma (\lambda)$ and
$\nu(\lambda)$ satisfy the above conditions (i)-(iv), then there
exists a unique $u\in\fn$ such that
$[u]\sp{\lambda}=\mathopen[\gamma(\lambda),\nu(\lambda)\mathclose]$
for each $\lambda\in\inter$.
\end{theorem}

As usual, given $u,v\in \fn$ and $k\in \mathbb{R}$, we define the sum $u+v:=\interu+\interv$ and the product $ku:=k\interu$.
It is well-known that $\fn$ endowed with these two natural operations is not a vector space, not even $(\fn,+)$ is a group.

The fuzzy number space $\fn$ can be endowed with several metrics (see, e.g., \cite {DK2}) but perhaps the most used is the following:

\begin{definition}{\rm \cite{GV,DK2}}\label{GW2}
 For $u, v\in\fn$, 

$$
d\sb{\infty}(u,v):=\sup\sb{\lambda\in\inter}\max\left\{|u\sp{-}(\lambda)-v\sp{
- } (\lambda)|, |u\sp{+}(\lambda)-v\sp{+} (\lambda)|\right \}.
$$

\end{definition}

It is called the supremum
metric on $\fn$, and with this metric, $\fn$ is a complete metric
space. Indeed, $\R$ endowed with the Euclidean topology can be topologically identified
with the closed subspace $\tilde{R}=\cj{\tilde{x}}{x\in\R}$ of
$(\fn , d\sb{\infty})$ where
$\tilde{x}\sp{+}(\lambda)=\tilde{x}\sp{-}(\lambda)=x$ for all
$\lambda\in\mathopen[0,1\mathclose]$. We will always
consider $\fn$ equipped with the supremum metric.

\begin{proposition} \label{properties} (\cite[Proposition 2.3]{FSS17}) The metric space $(\fn, d_\infty)$ satisfies the following properties:
\begin{enumerate}
\item $d_\infty ( \sum_{i=1}^{m}u_{i}, \sum_{i=1}^{m}v_{i}) \le \sum_{i=1}^{m}d_{\infty} (u_i, v_i)$ where $u_i,v_i\in \mathbb{E}^1$ for $i=1,...,m$.
\item  $d_\infty ( ku, kv) = k d_\infty (u, v)$ where $u,v\in \mathbb{E}^1$ and $k>0$.
\item $d_\infty (ku,\mu u ) = \mid k - \mu  \mid d_\infty (u,0),$ where $u \in \mathbb{E}^1$, $ k \ge 0$ and $\mu \ge 0$.
\item $d_\infty (ku,\mu v ) \le \mid k - \mu  \mid d_\infty (u,0) + \mu d_\infty ( u,v),$ where $u,v \in \mathbb{E}^1$, $ k \ge 0$ and $\mu \ge 0$.
\end{enumerate}
\end{proposition}

%
%
%

\bigskip

In the space of continuous functions defined between the compact Hausdorff space $K$ and the metric space $(\mathbb{E}^1, d_\infty)$, i.e., $C(K,\mathbb{E}^1)$, we will consider the following metric:
$$D(f,g) =  \sup_{ t \in K} d_\infty(f(t),g(t)),$$
which induces the uniform convergence topology on $C(K,\mathbb{E}^1)$.


%
%


Let us next introduce a basic tool for our purposes.

\begin{definition} \label{multi} Let $W$ be a nonempty subset of $C(K, \mathbb{E}^1)$. We define
$$Conv(W):=\{ \varphi \in C(K, [0,1]): \varphi f + (1 - \varphi) g \in W \hspace{0,2cm} for \hspace{0,2cm} all \hspace{0,2cm} f,g \in W \}.$$
\end{definition}

\begin{proposition} \label{multiplier} (\cite[Proposition 3.2]{FSS17}) Let $W$ be a nonempty subset of $C(K, \mathbb{E}^1)$. Then we have:
\begin{enumerate}
\item $\phi \in Conv(W)$ implies that $1 - \phi \in Conv(W)$.
\item If $\phi, \varphi \in Conv(W)$, then $\phi \cdot \varphi \in Conv(W)$.
\end{enumerate}
\end{proposition}

%
%
%
%

\begin{definition}
It is said that $M \subset C(K, [0, 1])$ separates the points of $K$ if given $s, t \in K$, there exists $\phi \in M$ such that $\phi(s) \not= \phi(t)$.
\end{definition}

%
%

\begin{lemma}\label{pat3} (\cite[Lemma 3.6]{FSS17}) Let $ W \subseteq C(K, \mathbb{E}^1)$. If $Conv(W)$ separates the points of $K$, then, given $x_0 \in K$ and an open neighborhood $N$ of $x_0$, there exists a neighborhood $U$ of $x_0$ contained in $N$ such that for all $0 < \delta < \frac{1}{2}$, there is $\varphi \in Conv(W)$ such that

\begin{enumerate}
\item $\varphi (t) > 1 - \delta,$ for all $t \in U$;
\item $\varphi (t) < \delta,$ for all $t \notin N$.
\end{enumerate}

\end{lemma}

\section{Best approximation for subspaces of $C(K,\mathbb{E}^1)$}

\begin{definition}
Let $A$ be a subspace of $C(K,\mathbb{E}^1)$ and let $f\in C(K,\mathbb{E}^1)$. We define

$$d(f,A):=\inf_{g\in A}\{\sup_{x\in K}d_{\infty}(f(x),g(x))\}=\inf_{g\in A}\{D(f,g)\}$$

$$d_x(f,A):=\inf_{g\in A}\{d_{\infty}(f(x),g(x))\}$$
\end{definition}

\begin{theorem} \label{main} Let $W$ be a subspace of $C(K,\mathbb{E}^1)$ and assume that $Conv(W)$ separates points.
For each $f \in C(K,\mathbb{E}^1)$, we have
\[d(f,W)=d_x(f,W)\]
for some $x\in K$.
\end{theorem}

\begin{proof}
We shall first show that $d(f,W)=\sup_{x\in K}d_x(f,W)$. It is apparent that $d(f,W)\ge \sup_{x\in K}d_x(f,W)$ since $d(f,W)\ge d_x(f,W)$ for each $x\in K$.
Let us prove that $d(f,W)\le \sup_{x\in K}d_x(f,W)$.

To this end, fix $\varepsilon > 0$. Given $x'\in K$, we can find $f_{x'}\in W$ such that $d_{\infty}(f(x'),f_{x'}(x'))<\sup_{x\in K}d_x(f,W)+\varepsilon$.
Let us define the following open neighborhood of $x'$:
$$N(x'):= \{ t \in K: d_\infty (f(t), f_{x'}(t)) < \sup_{x\in K}d_x(f,W)+\varepsilon \}.$$
Let  $U(x')$ be an open neighborhood of $x'$ satisfying the properties in Lemma \ref{pat3}.

Since $K$ is compact, there exist finitely many $\{ x_1,	\ldots , x_m \}$ in $K$ such that $K \subset U(x_1) \cup \ldots \cup U(x_m)$.
Choose $\delta>0$ such that $\delta<min(1, \frac{\varepsilon}{km})$, where
\[k:=\max \{D(f,0), D(f, f_{x_1}),...D(f, f_{x_m})\}.\]
By Lemma \ref{pat3}, we know that there exist $\phi_1, \cdots, \phi_m \in Conv(W)$ such that for all $i=1, \ldots, m$,
\begin{enumerate}
\item[(i)] $\phi_i (t) > 1 - \delta,$ for all $t \in U(x_i)$;
\item[(ii)] $0 \le \phi_i (t) < \delta,\quad if  \enspace  t \notin N(x_i)$.
\end{enumerate}

Let us define the functions
\newline
$\psi_1 := \phi_1$,
\newline
$\psi_2 := ( 1 - \phi_1) \phi_2$,
\newline
$\vdots$
\newline
$\psi_m := ( 1 - \phi_1)( 1 - \phi_2) \cdots ( 1 - \phi_{m-1}) \phi_{m}$.

By Proposition \ref{multiplier}, we know that $\psi_i \in Conv(W)$ for all $i=1,\ldots,m$.
Next we claim that
$$\psi_1+\ldots+\psi_j= 1 - ( 1 - \phi_1)( 1 - \phi_2) \cdots ( 1 - \phi_j),$$ $j=1,\ldots,m .$
Indeed, it is clear that
$$\psi_1 + \psi_2 = \phi_1  + (1 - \phi_1) \phi_2= 1-(1-\phi_1)\cdot(1-\phi_2).$$
We proceed by induction. Assume that the result is true for a certain $j\in\{4,...,m-1\}$ and let us check
$$\psi_1+\ldots+\psi_j+\psi_{j+1}= 1 - ( 1 - \phi_1)( 1 - \phi_2) \cdots ( 1 - \phi_j) ( 1 - \phi_{j+1}).$$
Namely,
$$\psi_1+\ldots+\psi_j+\psi_{j+1}=$$
$$=1 - ( 1 - \phi_1)( 1 - \phi_2) \cdots ( 1 - \phi_j)+( 1 - \phi_1)( 1 - \phi_2) \cdots ( 1 - \phi_{j}) \phi_{j+1}=$$
$$=1 - ( 1 - \phi_1)( 1 - \phi_2) \cdots ( 1 - \phi_j) ( 1 - \phi_{j+1}),$$
as was to be checked.

Fix $x_0\in K$. Then there is some $i_0\in \{1, \ldots, m\}$ such that $x_0\in U(x_{i_0})$. Hence $\phi_{i_0}(x_0)> 1 - \delta$ and consequently,
\[1\ge \sum_{i=1}^{m}\psi_i(x_0)=1-(1 - \phi_{i_0}(x_0)) \prod_{i \not= i_0} ( 1 - \phi_i(x_0))> 1- \delta.\]

Furthermore, we clearly infer
\begin{equation}\label{1}
\psi_i(t) < \delta \hspace{0.1in} {\rm for} \hspace{0.04in} {\rm all} \hspace{0.04in} t \notin N(x_{i_0}), \hspace{0.04in}i = 1, \ldots, m.
\end{equation}

Let
\begin{equation}\label{g}
h := \psi_1 f_{x_1} + \psi_2 f_{x_2} + \ldots + \psi_m f_{x_m}.
\end{equation}
It seems apparent that
$$h= \phi_1 f_{x_1} + ( 1 - \phi_1 ) [\phi_2 f_{x_2} + (1 - \phi_2) [\phi_3 f_{x_3} + \cdots + (1 - \phi_{m-1}) [\phi_m f_{x_m} \cdots]].$$
Therefore, $h \in W$ since $\phi_i\in Conv(W)$ for $i=1,...,m$ (see Definition \ref{multi}).

From Proposition \ref{properties}, we know that, given $x \in K$,
$$d_\infty (f(x), h(x)) \le d_\infty \left(f(x),\sum^m_{i=1} \psi_i (x)f(x)\right) + d_\infty \left(\sum^m_{i=1} \psi_i (x)f(x),h(x)\right)\le $$
$$\le \left|1-\sum^m_{i=1}\psi_i (x)\right|d_\infty \left(f(x), 0)\right) + \sum^m_{i=1}\psi_i (x) d_\infty \left( f(x),f_{x_i}(x)\right).$$

On the one hand, $\left|1-\sum^m_{i=1}\psi_i (x)\right|d_\infty \left(f(x), 0)\right)\le \delta D(f,0) \le \epsilon$.

On the other hand, let $$I_x=\{ 1 \le i \le m: x \in N(x_i) \}$$ and $$J_x=\{ 1 \le i \le m: x \notin N(x_i) \}.$$
Then, for all $i \in I_x$, we have
$$\psi_i (x) d_\infty (f(x), f_{x_i}(x)) \le \psi_i(x) (\sup_{x\in K}d_x(f,W)+\varepsilon)\le \sup_{x\in K}d_x(f,W)+\varepsilon$$
and, for all $i \in J$, the inequality (\ref{1}) yields
$$\psi_i(x) d_\infty (f(x),f_{x_i}(x)) \le \delta d_\infty (f(x),f_{x_i}(x))\le \delta D(f, f_{x_i})\le \delta k.$$

From the above two paragraphs, we can infer
$$d_\infty (f(x), h(x)) \le \epsilon + \sup_{x\in K}d_x(f,W)+\varepsilon + \delta k m \le \sup_{x\in K}d_x(f,W)+3\varepsilon$$
and, since $x\in K$ is arbitrary,
$$D(f,h) \le \sup_{x\in K}d_x(f,W)+3\varepsilon.$$
As a consequence, we deduce that
$$d(f,W)=\inf_{g\in W}\{D(f,g)\}\le \sup_{x\in K}d(f_x,W_x).$$

Finally, we can define a continuous function $\gamma: K\longrightarrow \mathbb{R}$ as
\[\gamma(x):=\inf_{g\in W}\{d_{\infty}(f(x),g(x))\}.\]
Since $K$ is compact, we know that $\gamma$ attains its supremum at some $x'\in K$. Hence, we can write
\[d(f,W)=d_{x'}(f,W).\]

\end{proof}

\section{Best approximation with respect to real-valued continuous functions}

In Approximation Theory, a natural question is when we can approximate a set-valued function by continuous real-valued functions. In the classical setting, Cellina's Theorem (\cite{cellina}) is the fundamental result (see also \cite{Beer}, \cite{blasi}, \cite{hola}).  In this section we introduce a method to measure the distance between a fuzzy-number-valued function and a real-valued one. Then we will prove the existence of the best approximation of a fuzzy-number-valued continuous functions to the space of real-valued continuous functions.

\begin{definition}\label{def:dist1}
Let $f \in C(K,\mathbb{E}^1)$ and let $F\in C(K)$. The distance between $f$ and $F$ can be defined as
\[D(f,F):=\sup_{x\in K}\left\{\sup\{|F(x)-t|:t\in [f(x)^-(1),f(x)^+(1)]\}\right\}.\]
\end{definition}

\begin{remark}
Indeed we are evaluating the distance between $f$ and $F$ at the core of each fuzzy number $f(x)$. We can define a similar distance at each level $\lambda\in[0,1]$ as follows

\[
D_{\lambda}(f,F):=\sup_{x\in K}\left\{\sup\{|F(x)-t|:t\in I_{\lambda}:=[f(x)^-(\lambda),f(x)^+(\lambda)]\}\right\}.
\]
and take $D(f,F)=\inf_{\lambda\in[0,1]}D_{\lambda}(f,F)$.
But, being the intervals $I_\lambda$ a nonincreasing family, it implies that the distances form a nondecreasing family as well, that is,
\[
D_{\lambda}(f,F)\geq D_{\eta}(f,F) \quad \text{if $0\leq\lambda\leq \eta\leq 1$}.
\]
Then $D(f,F)=D_{1}(f,F)$ which coincides with Definition~\ref{def:dist1}.
To avoid the case when the core reduces to a real number, we could define
$D(f,F)=\sup_{\lambda\in[0,1]}D_{\lambda}(f,F)=D_{0}(f,F)$, which means we are evaluating the distance at the support of each $f(x)$.
In that case, the following results could be  adapted easily to that situation.
\end{remark}

\begin{definition}
Let $f \in C(K,\mathbb{E}^1)$. We can define the distance between $f$ and $C(K)$ as
\[D(f,C(K)):=\inf_{F\in C(K)} D(f,F).\]
\end{definition}

\begin{definition}
Let $f \in C(K,\mathbb{E}^1)$ and $x\in K$. We can define
\[
rad(x,f):=\inf_{\alpha\in \mathbb{R}}\left\{\sup\{ |\alpha-\beta| \ :\ \beta\in [f(x)^-(1),f(x)^+(1)]\}\right\}
\]

It is clear that $rad(x,f)$ turns out to be the radius of the interval $[f(x)^-(1),f(x)^+(1)]$.
\end{definition}

\begin{definition}
Let $f \in C(K,\mathbb{E}^1)$. We define the {\em radius} of $f$ as
\[
rad(f):=\sup_{x\in K}rad(x,f)
\]
\end{definition}

\begin{remark}
From these definitions we infer easily that
\[D(f,F) \ge rad(f)\]
for all $F\in C(K)$. Hence
\[D(f,C(K)) \ge rad(f).\]
\end{remark}

%

\begin{theorem} Let $f \in C(K,\mathbb{E}^1)$. Then there exists a function $F_0\in C(K)$ such that
\[D(f,C(K))=D(f,F_0).\]
\end{theorem}

\begin{proof}
Let us define a map $G:K\longrightarrow 2^{\mathbb{R}}$ such that, for each $x\in K$,
\[G(x):=\{\alpha\in \mathbb{R}:[f(x)^-(1),f(x)^+(1)]\subseteq [\alpha-rad(f),\alpha+rad(f)]\}\]

Let us first check that $G(x)\neq \emptyset$ for each $x\in K$. We know that
\[
rad(x,f):=\inf_{\alpha\in \mathbb{R}}\sup_{\beta\in [f(x)^-(1),f(x)^+(1)]}|\alpha-\beta|\le rad(f).
\]
Since $rad(x,f)$ turns out to be the radius of the interval $[f(x)^-(1),f(x)^+(1)]$, it is clear that the center of this interval belongs to $G(x)$.

\medskip
It is apparent that $G(x)$ is closed  for each $x\in K$ since the intervals which appear in its definition are closed.

\medskip
Now we take $\alpha_1,\alpha_2$ in $G(x)$ and $k_1,k_2\ge 0$ with $k_1+k_2=1$. Then, given $\alpha\in [f(x)^-(1),f(x)^+(1)]$,
\[|\alpha-(k_1\alpha_1+k_2\alpha_2)|=|\alpha(k_1+k_2)-(k_1\alpha_1+k_2\alpha_2)|\le\]
\[\le k_1|\alpha-\alpha_1|+k_2|\alpha-\alpha_2|\le rad(f),\]
which shows that $G(x)$ is convex for each $x\in K$.

%
\medskip
Next, we shall prove that the map $G$ is lower semicontinuous, that is, we have to check that the set
\[\mathcal{O}:=\{x\in K: G(x)\cap O\neq \emptyset\}\]
is open in $K$ for every open set $O\subset \mathbb{R}$. To this end, fix $x_0\in \mathcal{O}$ and take $\alpha_0\in G(x_0)\cap O$ for a certain open set $O$.
Let $\delta_0>0$ such that $(\alpha_0-\delta_0,\alpha_0+\delta_0)\subset O$.
From the continuity of $f$ and from the fact that
\[
[f(x_0)^-(1),f(x_0)^+(1)]\subset [\alpha_0-rad(f),\alpha_0+rad(f)],
\]
 we infer that, given $\epsilon>0$, as we have
 \[
[\alpha_0-rad(f),\alpha_0+rad(f)]\subset (\alpha_0-rad(f)-\epsilon,\alpha_0+rad(f)+\epsilon),
 \]
there exist an open neighborhood, $Q(\epsilon)$, of $x_0$ such that
\[
[f(x)^-(1),f(x)^+(1)]\subset (\alpha_0-(rad(f)+\epsilon),\alpha_0+rad(f)+\epsilon)
\]
for all $x\in Q(\epsilon)$. Our goal is to prove that $Q(\epsilon)\subseteq \mathcal{O}$ for some $\epsilon>0$ to get the openness of $\mathcal{O}$.

Fix $x_1\in Q(\epsilon)$. Since $G(x_1)\neq \emptyset$, there exists $\alpha_1\in G(x_1)$. That is,
\[
[f(x_1)^-(1),f(x_1)^+(1)]\subset [\alpha_1-rad(f),\alpha_1+rad(f)].
\]
Besides, we know that
\[
[f(x_1)^-(1),f(x_1)^+(1)]\subset (\alpha_0-(rad(f)+\epsilon),\alpha_0+rad(f)+\epsilon).
\]

Taking $\epsilon>0$ as small as necessary, we can find $\delta'<\delta_0$ such that
\begin{align*}
&[\alpha_0-(rad(f)+\epsilon),\alpha_0+rad(f)+\epsilon]\cap [\alpha_1-rad(f),\alpha_1+rad(f)]\\[1ex]
&\subset [(\alpha_0+\delta')-rad(f),(\alpha_0+\delta')+rad(f)]
\end{align*}
and consequently
\[
[f(x_1)^-(1),f(x_1)^+(1)]\subset [(\alpha_0+\delta')-rad(f),(\alpha_0+\delta')+rad(f)],
\]
which implies that $\alpha_0+\delta'\in G(x_1)\cap O$. Hence $x_1\in  \mathcal{O}$, as desired.

\bigskip
Gathering the information we have obtained so far , we know that $G$ is a lower semicontinuous mapping defined between $K$ and the closed convex subsets of $\mathbb{R}$. Hence, by the Michael Selection Theorem (\cite{Michael}), we infer that there exists $F_0\in C(K)$ such that $F_0(x)\in G(x)$ for all $x\in K$.

As a consequence of the above paragraph, we can deduce that
\[D(f,F_0) \le rad(f),\]
which combined with the comments before this theorem yields
\[D(f,F_0) = rad(f)=D(f,C(K)).\]

\end{proof}
\subsection*{Conclusions}

In this paper we address two problems of best approximation type in the context of fuzzy-number-valued continuous functions: (1) the problem of the uniform approximation of a fuzzy-number-valued
function continuous on a compact set by a set of other functions continuous on this compact set; and (2) the existence of the best approximation of a fuzzy-number-valued continuous function to the space of real-valued continuous functions. We obtain positive results in both cases.

\end{document}